\documentclass[12pt,reqno,a4paper]{amsart}
\usepackage{extsizes}
\usepackage{blindtext}
\usepackage[margin=27mm]{geometry}
\usepackage{amssymb}
\usepackage{amsmath}
\usepackage{color}
\usepackage[all]{xy}
\usepackage[OT2,T1]{fontenc}
\usepackage{bm}
\usepackage{thmtools, thm-restate}
\usepackage{pdfsync}
\usepackage{hyperref}
\usepackage{mathrsfs}
\usepackage{dsfont}
\usepackage{textcomp}
\usepackage{tipa}
\usepackage{mathtools}
\usepackage[english]{babel}
\usepackage[autostyle]{csquotes}
\usepackage{xcolor}
\usepackage{enumerate}
\usepackage{upgreek}
\usepackage{tikz}

\usepackage{amscd}

\usepackage{caption,booktabs}
\usepackage[]{amsrefs}
\usepackage{setspace}
\usepackage{outlines}
\usepackage[flushleft]{threeparttable}
\usepackage{cleveref}

\newcommand{\Mod}[1]{\ (\mathrm{mod}\ #1)}

\newcommand{\Q}{\mathbb{Q}}

\newcommand{\Z}{\mathbb{Z}}

\newcommand{\PP}{\mathbb{P}}

\newcommand{\BZ}{\mathbb{Z}}

\DeclareMathOperator{\Gal}{Gal}

\newtheorem{lem}{Lemma}[section]

\newtheorem{prop}[lem]{Proposition}

\newtheorem{claim*}{Claim}
\newtheorem{thm}[lem]{Theorem}

\theoremstyle{rem}
\newtheorem{rem}[lem]{Remark}

\title{The Effect of Quadratic Base Change on Torsion of Elliptic Curves}

\author{{\i}rmak Bal\c{c}{\i}k}
\author{Burton Newman}
\address{Northwestern University, Department of Mathematics, 2033 Sheridan Road, Evanston, IL 60208, USA}
\email{irmak.balcik@northwestern.edu}
\email{bnewman235@gmail.com}
\address{}
\email{}

\begin{document}

\begin{abstract}
    Let $K$ be a quadratic number field and let $E$ be an elliptic curve defined over $K$ such that $E[2] \not\subseteq E(K).$ In this paper, we study the effect of quadratic base change on $E(K)_{\text{tor}}.$ Moreover, for a given elliptic curve $E/K$ with prescribed torsion group over $K,$ (no restriction on its $2$-torsion part) we describe a fast algorithm to find all quadratic extensions $L/K$ in which $E(K)_{\text{tor}} \subsetneq E(L)_{\text{tor}}$ and describe $E(L)_{\text{tor}}$ in each such case. In particular, we determine the growth of $E(K)_{\text{tor}}$ upon quadratic base change when $K$ is any quadratic cyclotomic field, which completes the earlier work of the second author \cite{New17}. 
\end{abstract}
\maketitle
\section{Introduction}
The possible torsion groups $E(\Q)_{\text{tor}}$ of an elliptic curve $E$ over $\Q$
are known by a celebrated theorem of Mazur \cite{Maz78}. These groups are the following:
\begin{equation} \label{eq:mazurgroups}
\begin{array}{ll}
C_n & 1\leq n \leq 12,\  n\neq 11 \\
C_2 \oplus C_{2n} & 1\leq n\leq 4.
\end{array}
\end{equation}

Subsequently, the work of Mazur was generalized to quadratic number fields by Kamienny \cite{Kam92} and Kenku-Momose \cite{KM88}. The possible torsion groups that can be realized as $E(K)_{\text{tor}}$ as one varies $E$ over all elliptic curves defined over any quadratic number field $K$ are the following:
\begin{equation} \label{eq:kamiennygroups}
\begin{array}{ll}
 C_n &  1\leq n\leq 18,\ n\neq 17  \\  
 C_2 \oplus C_{2n} & 1\leq n\leq 6  \\ 
C_3 \oplus C_{3n} & 1\leq n\leq 2 \\ 
C_4 \oplus C_4. &  
\end{array}
\end{equation}

The full description of torsion groups appearing over cubic number fields has been settled in \cite{DEVMB21}. Over quartic number fields, we still do not have a complete classification, analogous to $\eqref{eq:mazurgroups}.$ We know however \cite{JKP06} which torsion groups occur infinitely often up to isomorphism and \cite{DKSS17} that $17$ is the largest prime dividing the order of a torsion group over a quartic number field. As a first step toward this direction, we ask ``given an elliptic curve $E$ over any quadratic number field $K$, if $E[2] \not\subseteq E(K)_{\text{tor}}$, then how does $E(L)_{\text{tor}}$ relate with $E(K)_{\text{tor}}$ where $[L:K]=2$?" If $E(K)_{\text{tor}} \subsetneq E(L)_{\text{tor}},$ we say that torsion grows. In particular, we know $E(L)_{\text{tor}} = E(K)_{\text{tor}}$ for all but finitely many quadratic extensions $L/K$ since there are only finitely many torsion groups that can arise over a quartic number field. The case $K=\Q$ has already been studied extensively in \cite{GJT14}, \cite{GJT15}, \cite{JNT},  \cite{Gonz_lez_Jim_nez_2017}, \cite{Gonz_lez_Jim_nez_2019}, \cite{Gonz_lez_Jim_nez_2020}. We studied the main objective, which establish the first main result of this paper below. In addition, given any elliptic curve $E/K$ and a complete list of all possible torsion structures over $K,$ we describe an algorithm to find all quadratic extensions $L/K$ in which $E(K)_{\text{tor}}$ grows and describe $E(L)_{\text{tor}}$ in each such case. One can find the informal description of the algorithm in section 6 and the Magma \href{https://github.com/BaseChangeTorsion/The-Effect-of-Quadratic-Base-Change-on-Torsion}{{\bf{code}}} for its implementation.

\begin{thm}\label{mainthm1}
Let $K$ be a quadratic number field, $E/K$ an elliptic curve, and let $L/K$ be any quadratic extension with $L=K(\sqrt{d})$ for $d \in K$ a non-square. 
\begin{enumerate}[(i)]
\item \label{2p-2x4} If $E(K)[2^{\infty}]\simeq C_2$, then $C_2 \oplus C_4 \not\subseteq E(L).$ 

\item \label{2p^k} If $C_{2p^k} \subseteq E(L)$ for an odd prime $p$ and $k>0,$ then either $E(K)$ or $E^d(K)$ has a subgroup of the form $C_{2p^k}.$

\item \label{4-4x8} If $E(K)_{\text{tor}}\simeq C_4$ then $ E(L)_{\text{tor}} \not\simeq C_4 \oplus C_8.$

\item \label{4-16} If $E(K)_{\text{tor}}\simeq C_4$ then $E(L)_{\text{tor}} \not\simeq C_{16}.$

\item \label{8-2x16} If $E(K)_{\text{tor}}\simeq C_8$ and $E(L)_{\text{tor}}\simeq C_2 \oplus C_{16}$ then $E^d(K)_{\text{tor}} \simeq C_4.$

\item \label{2-2x20} If $E(K)[2]\simeq C_2,$ then $C_2 \oplus C_{20} \not\subseteq E(L).$ 

\item \label{4-2x24} If $E(K)_\text{tor}\simeq C_4$, then $E(L)_\text{tor} \not\simeq C_2 \oplus C_{24}$.

\item \label{16point} If $C_{32} \subseteq E(L)$ then either $E(K)$ or $E^d(K)$ has a point of order $16.$
\end{enumerate}
\end{thm}

Theorem \ref{mainthm1} allows to establish the second main theorem of this paper for which one can find Table \ref{table:3.2} for every possible situation.

\begin{thm}\label{mainthm2}
Let $K$ be a quadratic cyclotomic field, $E$ an elliptic curve defined over $K$ with $E(K)[2]\simeq C_2,$ and let $L/K$ be any quadratic extension.  
\begin{enumerate}
\item If $K=\Q(i),$ then $E(L)_{\text{tor}}$ is isomorphic to one of the following groups:
\begin{equation*} 
\begin{array}{ll}
    C_{2n} & 1\leq n \leq 8,\ n \neq 7 \\
    C_2 \oplus C_{2n} & 1 \leq n \leq 8,\ n\neq 7\\
    C_3 \oplus C_{6}. &   
\end{array}
\end{equation*}
\item If $K=\Q(\sqrt{-3}),$ then $E(L)_{\text{tor}}$ is isomorphic to one of the groups listed above or
$ C_4 \oplus C_4. $
\end{enumerate}
\end{thm}

The case with trivial 2-torsion is treated differently mainly due to Lemma \ref{lem1}, which establishes the third main result of this paper. 

\begin{thm}\label{burtonmainthm}
    Let $K$ be a quadratic cyclotomic field, $E$ an elliptic curve defined over $K$ with $E(K)[2]\simeq C_1,$ and let $L/K$ be any quadratic extension. 
    \begin{enumerate}
    \item If $K=\Q(i),$ then $E(L)_{\text{tor}}$ is isomorphic to one of the following groups:
    \begin{equation*} 
\begin{array}{ll}
    C_{n}, & n \in \{1,3,5,9,15\}.
\end{array}
\end{equation*}
\item If $K=\Q(\sqrt{-3}),$ then $E(L)_{\text{tor}}$ is isomorphic to one of the groups listed above or
$ C_3 \oplus C_3.$ 
\end{enumerate}
\end{thm}

In addition, these groups are realized except possibly the group $C_2 \oplus C_{16}$ (see Tables \ref{table:last2}, \ref{table:last1}). Compared to the $K=\Q$ case, we used different techniques, ranging from the Galois action on torsion subgroup of an elliptic curve to symmetric square of an algebraic curve to study quadratic points on certain modular curves (see Theorem \ref{isogeny}).

Two reasons motivate our focus on these fields. First, the classification of possible torsion groups for each quadratic cyclotomic field is complete due Najman (see Theorem \ref{mainlist} below). Second, the study of growth of torsion for elliptic curves without full $2$-torsion is the remaining open case in the second author's earlier result \cite{New17}. 

\begin{thm}[\cite{Naj10}, \cite{naj111}]\label{mainlist} Let $K$ be a quadratic cyclotomic field and let $E$ be any
elliptic curve defined over $K.$
\begin{enumerate}
\item If $K=\Q(i)$, then $E(K)_{\text{tor}}$ is either one of the groups in \eqref{eq:mazurgroups} or $C_4 \oplus C_4.$

\item If $K=\Q(\sqrt{-3})$, then $E(K)_{\text{tor}}$ is either one of the groups in \eqref{eq:mazurgroups}, $C_3 \oplus C_3$ or $C_3 \oplus C_6.$
\end{enumerate}
\end{thm}

Note that $C_3 \oplus C_3$ and $C_3 \oplus C_6$ are only realized over $\Q(\sqrt{-3})$ and
$C_4 \oplus C_4$ is only realized over $\Q(i)$ since they contain a root of unity for $3$ and $4,$ respectively. 

\subsection*{Outline} In section 2, we collect the known results for number fields used in the sequel. In section 3, we prove Theorem \ref{mainthm1} for which we use the Gal$(L/K)$-action on $E(L)_{\text{tor}}.$ The Galois action puts certain constraints on $E(L)_{\text{tor}}$ depending only on the structure of $E(K)_{\text{tor}}.$ In section 4, given a quadratic cyclotomic field $K$, we study $K$-rational points lying on certain modular curves $X_0(N),$ which will play a key role in determining the list of possible torsion groups upon quadratic base change. In section $5$, we prove our main results Theorem \ref{mainthm2} and \ref{burtonmainthm}. In section 6, we study growth for an arbitrary elliptic curve $E$ over a quadratic number field $K$ (no restrictions on its $2$-torsion part). By relying upon Magma to factor division polynomials over quadratic number fields, as well as compute torsion over quartic number fields, we describe a general algorithm which determines all possible growths of any prescribed torsion group upon quadratic base change. More precisely, given an elliptic curve $E/K,$ and a complete list of all possible torsion groups appearing as $E(K)_{\text{tor}}$ (as in Theorem \ref{mainlist}), the algorithm determines all quadratic extensions $L/K$ in which $E(K)_{\text{tor}}$ grows and describe $E(L)_{\text{tor}}$ in each case. One can find the code available \href{https://github.com/BaseChangeTorsion/The-Effect-of-Quadratic-Base-Change-on-Torsion}{{\bf{here}}} for such an implementation in Magma. 

\subsection*{Notations} Given an elliptic curve $E$ over a number field $K$, let $\overline{K}$ denote its algebraic closure, $E[n]=\{P \in E(\overline{K}): nP=0\}$ the group of $n$-torsion points, $E(K)[p^{\infty}]$ the Sylow $p$-subgroup of $E(K)$ for prime $p,$ and $E^d$ the quadratic twist of $E$ by $d\in K.$ And the model defined by $E: y^2 = x^3+ax+b$ is denoted by $[a,b]$.

\subsection*{Acknowledgement} We are indebted to Andreas Schweizer for allowing us to use the results from his ongoing project with the first author and a multitude of e-mail responses during the preparation of this manuscript. We are grateful to Michael Stoll for his help on mathoverflow regarding Proposition \ref{21-cycle} as well as thankful to the referees for their valuable comments.

\section{Auxiliary Results}

We begin by setting the necessary background to understand the results of this paper. The next result is due Kwon, which allows us to bound the quadratic growth of torsion over number fields. To state his result, we will make remarks and introduce some notation.  

\begin{rem}
    Let $K$ be a number field and let $E$ be an elliptic curve over $K$ given by $y^2=x^3+Ax+B$ with $A,B \in K.$ If $L=K(\sqrt{d})$ is a quadratic extension of $K$ with non-square $d\in K,$ then $\Gal(L/K)$ that is generated by $\sigma,$ acts on $E(K)_{\text{tor}}$ in an obvious way. For $P \in E(L)_{\text{tor}},$ let $Q=(x,y)$ be the point $P-\sigma(P).$ If $P \in E(K)_{\text{tor}}$ then $Q=0.$ Otherwise it follows from 
$(\sigma(x),\sigma(y))=\sigma(Q)=-Q=(x,-y)$
that $x,y/\sqrt{d}$ are in $K$ and so $(x,y/\sqrt{d})$ lies on its quadratic twist $E^d$ given by $dy^2=x^3+Ax+B$. Note that $P-\sigma(P)$ in $E^d(K)$ is understood as $(x,y/\sqrt{d})$ in the proof of the following proposition.
\end{rem}

\begin{prop}(\cite{Kwo97})\label{injection}
Let $K$ be a number field, $L=K(\sqrt{d})$ for $d \in K$ a non-square, and let $\sigma$ denote the generator of $Gal(L/K)$. There is a homomorphism $h$ defined by
\begin{align*}
 E(L)_{\text{tor}} &\xrightarrow{h} E^d(K)_{\text{tor}} \\
 P &\mapsto P- \sigma(P)
\end{align*}
with ker$(h)=E(K)_{\text{tor}}$ and it induces an injection 
$E(L)_{\text{tor}}/E(K)_{\text{tor}} \hookrightarrow E^{d}(K)_{\text{tor}}.$
\end{prop}

\begin{proof}
See the proof when $K=\Q$ in \cite[Proposition 1]{Kwo97}.
\end{proof}

It is already known that the odd-order torsion part of $E(L)_{\text{tor}}$ can be well-understood by only studying two torsion groups which occur over $K.$

\begin{lem}(\cite{GJT14})\label{lem1}
	If n is an odd positive integer we have
	$$ E(K(\sqrt{d}))[n] \simeq E(K)[n] \oplus E^d(K)[n]. $$
\end{lem}


The following theorem lists various restrictions on growth in quadratic extensions.  

\begin{thm}\label{burton} 
    Let $K$ be a number field, $E/K$ an elliptic curve, L a quadratic extension of $K$ and $p$ an odd prime.
    \begin{enumerate}[(i)]
        \item \label{2-part} If $E(K)[2]$ is trivial, then $E(L)[2]$ is trivial.
        \item \label{2-ptTwist} If $d\in K$, $d \neq 0$ then $E(K)[2]\simeq E^d(K)[2].$
        \item \label{3x3} If $E(K)[p]$ is trivial and $E(L)[p]=C_p \oplus C_p$ then $K$ contains a primitive $p$th root of unity.
        \item If $E(K)[p]\simeq C_p$ and $E(L)[p^{\infty}] \neq E(K)[p^{\infty}]$ then $E(L)[p] \simeq C_p \oplus C_p.$
        \item \label{p-pxp} If $E(K)[p]\simeq C_p$ and $E(L)[p] \simeq C_p \oplus C_p$ then $K$ does not contain a primitive $p$th root of unity. 
        \item If $E(K)[p] \simeq C_p \oplus C_p$ then $E(L)[p^{\infty}]=E(K)[p^{\infty}].$
    \end{enumerate}
\end{thm}

\begin{proof}
    Parts $(i)$ and $(ii)$ are easily verified.

$(iii)$ Suppose $E(K)[p]$ is trivial.
By Lemma \ref{lem1}, it follows that $E^d(K)[p] \simeq C_p \oplus C_p$ so by the Weil pairing 
we conclude that $K$ contains a primitive $p$th root of unity.

$(iv)$ Let $m$ the largest positive integer such that there is an element of order $p^m$ in $E(L)_{tor}$. We have   $E(L)[p^{m}]= E(K)[p^{m}] \oplus E^d(K)[p^{m}]$ by Lemma \ref{lem1}. If 
$E(L)[p^{m}]\ne E(K)[p^{m}]$ then $E^d(K)[p^{m}] \not \simeq C_1$ so $E^d(K)[p]$ $\not \simeq C_1$.  Hence $E^d(K)[p] \simeq C_p$ or $C_p \oplus C_p.$ 
In the latter case this would yield $E(L)[p] \simeq C_p \oplus C_p \oplus C_p$ which gives a contradiction, 
so $E(L)[p] \simeq C_p \oplus C_p $.

$(v)$ Let $\mu_p$ be a primtive $p$th root of unity. Suppose $E(K) \simeq C_p$, and $E(L)[p] \simeq C_p \oplus C_p$.  Let $\sigma  \in $ Gal($L/K$) be nontrivial.  We can choose a basis for $E(L)[p]$ such that the induced Galois representation satisfies
$$ \rho: \mbox{Gal(L/K)} \rightarrow \mbox{Gl}_2(\BZ / p \BZ)$$
$$\sigma \mapsto 
\begin{bmatrix}
1 & \alpha\\
0 & \chi 
\end{bmatrix}$$
for some $\chi \in (\BZ / p \BZ)^{*}$, $\alpha \in \BZ / p \BZ$. If $\mu_p \in K$ then $\mu_p = \sigma(\mu_p) = (\mu_p)^{det(\rho(\sigma))}=(\mu_p)^{\chi}$
so $\chi = 1$ mod $p$.  As $\sigma^2=1$, $(\rho(\sigma))^2=1$ so $2\alpha = 0$. As $p$ is odd, we  conclude $\alpha=0$, so $\rho(\sigma)$ is the identity.  This means $\sigma$ acts trivially on the p-torsion, so $E(K)[p] \simeq C_p \oplus C_p$, contradicting our hypothesis.

$(vi)$ Suppose $E(K)[p] \simeq C_p \oplus C_p$. By Lemma \ref{lem1}, if $E(L)[p^{\infty}] \ne  E(K)[p^{\infty}]$ then $E^d(K)[p^{\infty}] \not \simeq C_1$ so $E^d(K)[p]$ $\not \simeq C_1$.  Hence $C_p \oplus C_p \oplus C_p \subseteq E(L)[p]$, a contradiction.
\end{proof}

\section{Restrictions on Growth}

In this section, we focus on the effect of quadratic base change on $E(K)_{\text{tor}}$ when $E$ varies over all elliptic curves defined over a quadratic number field $K.$ The strategy is to view how Gal$(L/K)$ acts on $E(L)_{\text{tor}}$ when $[L:K]=2$ and to conclude that this forces certain properties on $E$ over $K$. For certain structures of $E(L)_{\text{tor}}$, one would get properties of $E$ over $K$ that would violate known results about elliptic curves over $K$, thus ruling out the possibility of that particular torsion group $E(L)_{\text{tor}}.$ Before proceeding further, we need a general lemma to narrow down certain growth in the proof of Theorem \ref{mainthm1}.

\begin{lem}\label{generallemma4}
Let $K$ be a number field and let $E/K$ be any elliptic curve with $E(K)[2] \simeq C_2.$ If $E(L)$ contains a subgroup of the form $C_2 \oplus C_4$ where $L=K(\sqrt{d})$ for $d \in K$ a non-square, then both $E$ and its quadratic twist $E^d$ have a $K$-rational $4$-torsion point.
\end{lem}

\begin{proof}
Suppose $C_2 \oplus C_4 \subseteq E(L)$ such that $C_2 \oplus C_4 = \langle Q, P \rangle$ where $Q$ is a point of order $2$ and $P$ is a point of order $4.$ Let $\sigma$ denote the non-trivial element of Gal$(L/K).$  

If $\sigma(P)=P$ or $-P$, then $\sigma(2P)=2P$, so $2P$ is the unique $K$-rational $2$-torsion point. Then $\sigma(Q)=Q+2P$ and so 
$\sigma$ maps  $P+Q$ to its inverse or to itself depending on $\sigma(P)=P$ or $-P$ from which the statement follows.
\par
If $\sigma(P)$ is different from $P$ and 
$-P$, then $\sigma(P)+P, \sigma(P)-P$ are both non-trivial. Note that $\sigma(P)+P \in E(K)[4]$ and  $\sigma(P)-P \in E^d(K)[4].$ Our aim is to show that both have exactly order $4.$ If they have order $2,$ then $\sigma(P)+P$ and $\sigma(P)-P$ are both the unique $K$-rational $2$-torsion point, and hence they must be equal, which gives the contradiction $2P=0$. Now assume that one of them has order $4.$ By arguing similarly, we may assume that $\sigma(P)+P$ has order $4$ but $\sigma(P)-P$ has order $2.$ Since $\sigma(P)\neq P,-P,$ this leaves only $\sigma(P)-P=Q$ or $Q+2P.$ In both cases, $\sigma(P)+P$ must have order $2$ which is a contradiction by assumption.  
\end{proof}

\subsection*{Proof of Theorem \ref{mainthm1}}
$(i)$ An immediate corollary of Lemma \ref{generallemma4}.

$(ii)$ It follows from Lemma \ref{lem1} that $E(K)$ or $E^d(K)$ contains a point of order $p^k.$ By Theorem \ref{burton} $E(K)$ must have a point of order $2,$ so does $E^d(K)$. Therefore, either $E$ or $E^d$ has a point of order $2p^k$ defined over $K.$

$(iii)$ Suppose $E(L)_{\text{tor}} \simeq C_4 \oplus C_8.$ By Lemma \ref{generallemma4} $E^{d}(K)_{\text{tor}} \simeq C_4$ or $C_8.$ On the one hand, we have $|E(L)_{\text{tor}}/E(K)_{\text{tor}}| = 8$ by assumption. On the other hand, $E(L)_{\text{tor}}/E(K)_{\text{tor}}$ is cyclic since it embeds into $E^{d}(K)_{\text{tor}}$ by Proposition \ref{injection}. This forces $E^{d}(K)_{\text{tor}}$ to be isomorphic to $C_8$. But we claim that $E(L)_{\text{tor}}/E(K)_{\text{tor}}$ cannot have a point of order $8$. In detail, it is equivalent to showing that its embedding $h(E(L)_{\text{tor}}/E(K)_{\text{tor}}) \subseteq E^d(K)_{\text{tor}}$ has no element of order $8$. Let $E(L)_{\text{tor}} = \langle Q,R\rangle $ where $Q$ is of order $4$ and $R$ is of order $8.$ Let Gal$(L/K)=\langle \sigma \rangle$ and let $P\in E(L)$ be any point of order $8.$ Then $P=a_1Q + b_1R$ and $\sigma(P)=a_2Q+b_2R$ where both $a_1,a_2$ are odd. But this implies $P-\sigma(P)$ has order at most $4$ which gives a contradiction.

$(iv)$ \cite[Theorem 5(iv)]{GJT14}.

$(v)$ Assume the hypothesis. Fix a $2$-torsion point $Q$ and a $16$-torsion point $P$ such that $E(L)_{\text{tor}}=\langle Q, P \rangle.$ Let $\sigma$ be the non-trivial element of Gal$(L/K).$

\textit{Case 1} : $2P$ is $K$-rational. Note that $\sigma$ maps $P$ to a $16$-torsion point. So,  $\sigma(P)=aP$ or $aP+Q$ where $a$ is odd. Since $P + \sigma(P)$ is fixed by $\sigma$, we have $\sigma(P)=aP$. Otherwise, $Q\in E(L)[2]$ would be $K$-rational. Moreover, $\sigma(2P)=2P$ which implies that $a= 1$ or $a=9$. If $a=1$, then $P \in E(K)$, a contradiction. If $a=9,$ then $P+Q$ is defined over $K$ and has order $16,$ a contradiction.

\textit{Case 2} : $2P$ is not $K$-rational. Then 
$E(K)[8]=\{bP + Q : b \in \{2,6, 10, 14\}\}$. So $4P$ and $12P$ are the $K$-rational $4$-torsion points. Hence the $K$-rational $2$-torsion point again is $8P$. Notice that $\sigma(P)=aP$ or $aP+Q$ where $a$ is odd. So $2P$ goes to $2aP$, and on the other hand to 
$$\sigma(2P)=\sigma(2P+Q) + \sigma(Q) = 2P +Q +Q+ 8P = 10P.$$ 
It follows that $a =5$ or $13$. If $\sigma(P)=5P$ or $13P$ then $P+Q+\sigma(P+Q)=14P$ or $6P$ is $K$-rational, respectively. But this is not possible since $2P$ is not $K$-rational. This leaves only the possibilities: $\sigma(P) = 5P + Q$ and $\sigma(P)=13P + Q$. 
For example, if $\sigma(P)=13P+Q$, we simply take $13P+Q$ as the new $P$, and call it $\tilde{P}$. Then $\sigma(\tilde{P})=5\tilde{P}+Q$. So we may assume
$\sigma(P)=5P+Q$ and $E(K)_{\text{tor}}=\langle 2P+Q \rangle.$ 

By the list in \eqref{eq:kamiennygroups} and Lemma \ref{generallemma4}, the only possibilities for $E^d(K)_{\text{tor}}$ are $C_4, C_8$ and $C_{16}.$ 
Note that there exists a short exact sequence defined as follows
\begin{align}\label{exact1}
 0 \rightarrow \text{ker}(\psi)\xrightarrow{i} E(L) \xrightarrow{\psi} E(K) \times E^d(K) \xrightarrow{\pi} \text{coker}(\psi) \rightarrow 0 
\end{align}
where $\psi$ maps $R = (x,y)$ to  
$$\psi(R) = (R+\sigma(R), \phi(R-\sigma(R),1/\sqrt{d}))\ \text{with}\ \phi(R,a)=(x,ay).$$ Restricting to the torsion part, ker$(\psi)$ is either trivial, $C_2$ or full 2-torsion. One can observe from the action of Gal$(L/K)$ that ker$(\psi)=\langle 8P \rangle.$ If $E^d(K)_{\text{tor}} \simeq C_8$ or $C_{16},$ then coker$(\psi)$ has exponent at least $4.$ But this contradicts \cite[Theorem 3]{GJT14}, proving the assertion.

$(vi)$ Applying Lemma \ref{lem1} and Lemma \ref{generallemma4}, $E(K)$ or $E^d(K)$ would have a $20$-torsion point, which gives a  contradiction by the list in \eqref{eq:kamiennygroups}.

$(vii)$ Suppose $E(K)_{\text{tor}}\simeq C_4$ but $E(L) \simeq C_2 \oplus C_{24}.$ Fix an $L$-rational $8$-torsion point $P$ and a $2$-torsion point $Q$ different from $4P$. Let $\sigma \in$ Gal$(L/K)$ be the non-trivial automorphism. Then $\sigma(P)$ is an $L$-rational point of order $8$, so is equal to one of the eight points: $aP$ or $aP+Q$ where $a$ is odd. This implies $\sigma(4P)=4P$, and hence $4P$ is the unique $K$-rational $2$-torsion point on $E$. Consequently, $\sigma(Q)=Q+4P$ from which one can observe that $\sigma(P)$ cannot be $aP+Q$; otherwise $\sigma$ would be an automorphism of order $4$. 

This leaves the case $\sigma(P)=aP$. If $a=1$, then $P$ is a $K$-rational $8$-torsion point, contradicting the assumption. If $a=5$, then $P+Q$ is fixed by $\sigma$ and it has order $8$, a contradiction. If $a=7$, then $E^d(K)$ contains a $8$-torsion point and also a $3$-torsion point by Lemma \ref{lem1}, so a $24$-torsion point, contradicting the list in \eqref{eq:kamiennygroups}. 
If $a=3$, then $P+Q$ is mapped its inverse and so we are back to the previous case.

$(viii)$ Let $C_{32}$ be a subgroup of $E(L)$. It follows from the Weil Pairing that $E(L)[32] \simeq C_M \oplus C_{32}$
where $M$ divides $8.$ Fix a generator $P$ of $C_{32}$ and a generator $Q$ of $C_M$. Let Gal$(L/K)=\langle \sigma \rangle$. Then, $\sigma(P)$ is a point of order $32$, so it is of the form $aP+bQ$ where $a$ is odd and $b$ could be $0$. 

If $a\equiv 1$ mod $4$, then $\sigma(P)+P$ is a point of order $16$ since $a+1$ is only divisible once by $2$ and the order of $Q$ is at most $8$. Note that $\sigma(P)+P$ is fixed by $\sigma,$ so is in $E(K).$ 
If $a\equiv 3$ mod $4$, then we consider the point $\sigma(P)-P \in E^d(K)$ that has order $16$ for the same reason above.

\section{$K$-rational points on $X_0(N)$}
There is an affine curve $Y_0(n)$ whose $K$-rational points classify isomorphism classes $[(E,C)]_K$ of pairs $(E,C)$ where $E/K$ is an elliptic curve and $C$  is a cyclic $\Gal(\overline{K}/K)$-invariant subgroup of $E(\overline{K})$ of order $N$, or an \textit{$N$-cycle}.  Two pairs $(E,C)$,$(E',C')$ are equivalent if and only if there is an isomorphism $f:E \rightarrow E'$ such that $f(C)= C'$.  By adding a finite number of points (called $\emph{cusps}$) to $Y_0(N)$ we obtain the projective curve $X_0(N)$.  The curve $X_0(N)$ has a model over $\mathbb{Q}$ and hence we have tools to study the set $X_0(N)(K)$ of all $K$-rational points on $X_0(N)$. 

\begin{lem}\label{n-cycle}
    Let $K$ be a quadratic field and $E/K$ an elliptic curve. If $F$ is a Galois extension of $K$ and $E(F)[n]\simeq C_n$ then $E$ has an $N$-cycle.
\end{lem}
\begin{proof}
    Let $\{P,Q\}$ be a $C_n$-basis for $E[n].$ We may assume that $E(F)[n]=\langle P \rangle.$ For any $\sigma \in \Gal(\overline{K}/K),$ we have $\sigma(P) \in E[F][n]$ since $F/K$ is Galois. By assumption $\sigma(P) \in \langle P \rangle$ and hence $\langle P \rangle$ is $\Gal(\overline{K}/K)$-invariant as desired. 
\end{proof}

In the case $j=0, 1728$ over quadratic number fields, we use the technique from \cite{Lem}.

\begin{lem}\label{count}
Let p be a prime and $E/F_p$ an elliptic curve with model $y^2=x^3+Ax+B$.

\begin{enumerate}

\item If $A=0$ (i.e. $j(E)= 0$) and $p \equiv 2 \Mod 3$  then $|E(F_p)| = p+1$ and $|E(F_{p^2})| = (p+1)^2$.

\item If $B=0$ (i.e. $j(E)=1728$) and $p \equiv 3 \Mod 4$  then $|E(F_p)| = p+1$ and $|E(F_{p^2})| = (p+1)^2$.
\end{enumerate}
\end{lem}

\begin{proof}

If $A=0$  and $p \equiv 2 \Mod 3$  then $|E(F_p)| = p+1$ by \cite[Prop 4.33]{Was}.  If $B=0$ and $p \equiv 3 \Mod 4$  then $|E(F_p)| = p+1$ by \cite[Thm 4.23]{Was}.  In either case, $|E(F_{p^2})| = p^2+ 1-(\alpha^2+\beta^2)$ by \cite[Thm 4.12]{Was}, where $\alpha$ and $\beta$ are roots of $x^2+p$.  Hence,  
\begin{align*}
|E(F_{p^2})| & = p^2+ 1-(\alpha^2+\beta^2)\\
& =  p^2+ 1 - (-p-p)\\
& = p^2+2p+1\\
& = (p+1)^2.
\end{align*}
\end{proof}

\begin{prop}\label{cm_case}
Let $K$ be a quadratic field and let $E/K$ be any elliptic curve. If $j(E) = 0$ and $p > 3$ is a prime then $E(K)_{\text{tor}}$ has no element of order $p$. If $j(E) = 1728$ and $p > 2$ is a prime, then $E(K)_{\text{tor}}$ has no element of order $p$.
\end{prop}

\begin{proof}
Suppose $j(E) = 0$. Twisting by a square in $\mathcal{O}_K$ if necessary, we may
assume that $E$ has a model of the form $y^2 = x^3 + Ax + B$ with $A, B \in \mathcal{O}_K$. Note
that since $\mathcal{O}_K$ is a Dedekind domain, the principal ideal $(\Delta(E))$ generated by the discriminant of $E,$ has only a finite number of prime ideal divisors, and hence $\Delta(E)$ lies in only a finite
number of prime ideals of $\mathcal{O}_K$. 

Let $q > 3$ be a prime in $\mathbb{Z}$. Since $q\neq 3,$ by
the Chinese remainder theorem there exists an integer $n$ satisfying
\begin{align*}
    n+1 & \equiv 2 \Mod q\\
    n  & \equiv 2 \Mod 3.
\end{align*}
Furthermore, $n+3qk$ satisfies the congruences above for every integer k, and $(n, 3q) = 1$
by the congruences above. Hence by Dirichlet's theorem on arithmetic progressions, there
are infinitely many primes in this arithmetic progression. In particular, there is a prime $p$ satisfying the congruences above such that $E$ has good reduction modulo a prime ideal $\beta$ above $p$. As $[K:\Q]=2,$ we have $\mathcal{O}_K/\beta \simeq \mathbb{F}_p$ or $\mathbb{F}_{p^2}.$ 
Moreover, we have an injection of the group $E(K)[\overline{p}]$ into $E(\mathbb{F}_{p})$ or $E(\mathbb{F}_{p^2}).$ But it follows from Lemma \ref{count} that
\begin{align*}
    |E(\mathbb{F}_p)| &= p+1 \equiv 2 \not\equiv 0 \Mod q  \\
    |E(\mathbb{F}_{p^2})|&=(p+1)^2 \equiv 4 \not\equiv 0 \Mod q
\end{align*}
as $q\neq 2$. Hence in either case (noting $p\neq q)$, we conclude there is no point of order $q$ in $E(K)_{\text{tor}}.$

Now suppose $j(E)=1728$. If $q$ is an odd prime, then one can argue just as in the $j(E)=0$ case
that there is no point of order $q.$
\end{proof}

Given an $N$-cycle $C,$ let $K_C$ denote the field of definition of $C$ (that is, the field obtained by adjoining to $K$ all the coordinates of the points of $C$) and for a polynomial $f$, let $K(f)$ denote the splitting field of $f$ over $K$. Magma provides the defining polynomial $f_C$ whose roots determine the $x$-coordinates of the points in $C.$ If $C$ is pointwise rational over a field $L,$ then $f_C$ should split completely over $L$. In particular, if $L$ is a quadratic extension of $K$, then $f_C$ must have irreducible factors of degree at most $2$ over $K$.

\section*{N=15}

The modular curve $X_0(15)$ is an elliptic curve with model
$$ y^2+xy+y=x^3+ x^2-10x-10.$$
Over $K=\Q(i),$ we see the first two entries in \Cref{x0(15)} indicate the only potential isomorphism classes in which we could find a 15-cycle pointwise rational over a quadratic extension of $K$. Hence if a pair $(E,C)$ exists with $E/K$, $C \subseteq E(\overline{K})$ cyclic $\Gal(\overline{K}/K$)-invariant and the points of $C$
$L$-rational for some quadratic extension $L/K$ then in fact $E$ is defined over $\Q$ and $C$ is $\Gal(\overline{\Q}/\Q)$-invariant.  

\begin{table}[ht]

 \caption{($K=\Q(i)$) Representatives $(E,C)$ of isomorphism classes corresponding to non-cuspidal K-rational points  on a model of $X_0(15)$ } 
      \begin{threeparttable}

 \centering 

    \begin{tabular}{ | l | l | l | l |}

    \hline
\label{x0(15)}

    Point & $j(E)$ &  $E$ &  Deg($f_C$) \\ \hline

    (8, -27) & -121945/32 & [-3915, 113670]
   
 &  (1,1,1,2,2) \\ \hline
    
     (-2, -2)& 46969655/32768 & [28485,-838890]
    
 & (1,1,1,2,2) \\ \hline
    
    (-13/4, 9/8)& -25/2 & [-675,-79650]
    
    &  (1,2,4) \\ \hline
    
   (3,-2)  & -349938025/8 & [-162675,-25254450]
   
   &  (1,2,4) \\ \hline
   
    $(1/2,(\mp 15i - 3)/4)$ & $(\mp 198261i-62613)/2$  & 
    $[\frac{\pm 6846i + 9528}{105625}, \frac{\mp 22652i + 30164}{2640625}] $&  (1,2,4)  \\ \hline

     $(\pm 3i-1, \mp6i+6)$  & $(\pm15363i - 47709)/256$  & $[\frac{\mp3i+96}{200}, \frac{\pm 3989i - 373}{10000}]$ &  (1,2,4) \\ \hline 
  
    $(\pm 3i-1, \pm 3i-6)$ &  $(\mp 13670181i+19928133)/8$ & $ [\frac{\pm 2583i + 9444}{8450}, \frac{\mp 93373i+39511}{211250} ] $  & (1,2,4) \\ \hline
    
     $(-7, \pm 15i+3)$ &  $(\mp 86643i-1971)/4$  &  $[\frac{\pm 216i-2688}{625}, \frac{\pm 8608i-53344}{15625} ]$ & (1,2,4) \\ \hline

\end{tabular}
  \begin{tablenotes}
            \item[$\dagger$]  In the last column we list the degrees of the irreducible factors of $f_C$ over K.
            
        \end{tablenotes}
     \end{threeparttable}
\end{table}

The point $(8,-27)$ corresponds to $(E,C)$ with 
$$
f_C =  (x - 7/10)(x+ 1/2)(x+ 17/10)(x^2+x- 139/20)(x^2+13x+ 269/20)
$$
A brief computation yields $K(f_C)=\Q(\sqrt{5})$. Since $j \not = 0, 1728$, any pair $(E',C')$ equivalent to $(E,C)$ is of the form $E'=E^d$, $C'=C^d$ for $d$ in $K$. As $K(f_C) = K(f_{C^d})$, if $K_C/K$ is degree 2 
then we must have $K_C = K(\sqrt{5})$. The point $(1/2, 3\sqrt{-6}/5  )$ is in $C$, so the only potential $d$-twists (up to a square in $K$) in which $K_{C^d}/K$ is degree 2 (namely $K(\sqrt{5})$) are $d=-6,-6 \cdot 5$.  Magma now tells us that for these two values of $d$, $E^d(K(\sqrt{5}))_{tor} \simeq C_{15}$. 

The point $(-2,-2)$ corresponds to $(E,C)$ with $f_C = h(x)q(x)$ where: 
$$h(x) =  (x - 3/104)(x+17/520)(x+113/520)$$
\vspace{-5mm}
$$q(x) = (x^2 - (11/52)x + 2333/54080)(x^2 + (1/52)x + 437/54080)$$
We compute by Magma that $K(f_C)=\mathbb{Q}(\sqrt{-15})$, so as above, if $K_C/K$ is degree 2 then we must have $K_C =
K(\sqrt{-15})$. The point $(3/104, 4\sqrt{26}/845  )$ is in $C$, so the only potential $d$-twists in which $K_{C^d}/K$ is degree 2,
are $d=26,26 \cdot (-15)$.  Magma now tells us that for these two values of $d$, $E^d(K(\sqrt{-15}))_{tor} \simeq C_{15}$.  Hence there are exactly four elliptic curves over $K$ (up to isomorphism over $K$) with a 15-cycle pointwise rational over a quadratic extension of $K$.

Similarly, over $K = \Q(\sqrt{-3})$ we find the same four elliptic curves which are the only elliptic curves over $K$ with a 15-cycle pointwise rational over a quadratic extension of $K$.

\section*{N=20}

The modular curve $X_0(20)$ is an elliptic curve with model $$ y^2=x^3+x^2+4x+4. $$

Over $K=\Q(\sqrt{-3}),$ $X_0(20)(K)$ has rank 0, torsion $C_6$ and these points are all cusps. Over $K=\Q(i),$ $X_0(20)(K)$ has rank $0$ and torsion $C_2 \oplus C_6$ with $6$ rational cusps. Using Magma, we compute the $4$ non-cuspidal $K$-rational points correspond to the isomorphism classes $[(E,C)]_{K}$ where $C$ is pointwise rational over an extension of $K$ of degree at least 4. The $2$ non-cuspidal $K$-rational points correspond to the isomorphism classes $[(E,C)]_{K}$ where $j(E)$ is equal to $1728$.

In case $j = 1728,$ we will argue by way of contradiction. If there exists an elliptic curve $E/K$ with $j(E)=1728$ and $C \subseteq E(L)$ a cyclic $\Gal(\overline{K}/K)$-invariant subgroup of order $20$ where $L=K(\sqrt{d})$ for $d \in K$ a non-square, then its quadratic twist $E^d$  has a point of order $5$ over $K$ by Lemma \ref{lem1}. But this contradicts Proposition \ref{cm_case}.

{
\renewcommand{\arraystretch}{1.3}

 \begin{table}[ht]
 \caption{($K=\mathbb{Q}(i)$) Representatives (E,C) of isomorphism classes corresponding to non-cuspidal K-rational points  on a model of $X_0(20)$   } 
  
  \begin{threeparttable}

\setlength{\tabcolsep}{14pt}

    \begin{tabular}{| l | l | l | l |}
    
    \hline
\label{x0(20)}    
    Point & j(E) &  E & $f_C$ \\ \hline
    
    $(\mp 2i, 0)$& $287496$ & $[\frac{\pm 264i + 77}{625}, \frac{\pm 616i + 1638}{15625}]$ & $(1,1,2,2,4)$ \\ \hline
    
   $(\pm 2i-2, \mp 2i-4)$  & $287496$ & $[\frac{\pm 264i + 77}{625}, \frac{\pm 616i + 1638}{15625}]$ & $(1,1,2,2,4)$ \\ \hline

   $(\pm 2i-2, \pm2i+4)$  & 1728 & See Proposition \ref{cm_case}   &   \\ \hline
   
    \end{tabular} 
     \begin{tablenotes}
            \item[$\dagger$]  In the last column we list the degrees of the irreducible factors of $f_C$ over K.
        \end{tablenotes}
     \end{threeparttable}

\end{table}
}

\section*{N=21}

The modular curve $X_0(21)$ is an elliptic curve with model
$$  
y^2 + xy = x^3 - 4x - 1 
$$.
Over $K=\mathbb{Q}(\sqrt{-3}),$ $X_0(21)$ has rank $0$ and torsion $C_2 \oplus C_8$ over $K$ with $4$ cusps.  The $12$ non-cuspidal points correspond to isomorphism classes $[(E,C)]_K$ and using Magma we found representatives of each class (see \Cref{x0(21)}).  As the table indicates, for each representative $(E,C)$ with $j \not = 0$, $f_C$ has an irreducible factor of degree at least $3$ and hence 
$[K_C:K] \geq [K(f_C):K] \geq 3$.
In particular, there is no quadratic extension $L/K$ such that all the points of $C$ are $L$-rational. Since $j \not = 0,1728$ in each of the aforementioned cases, 
the isomorphism class $[(E,C)]_K$ of $(E,C)$ only consists of $(E^d,C^d)$ for some non-zero $d \in K$, where $C^d$ denotes the image of $C$ under quadratic twist by $d$. If $(x,y) \in C$ then $(dx, d^{3/2}y) \in C^d$.  As $d \in K$, $K(f_C) = K(f_{C^d})$.  Hence the $8$ isomorphism classes with $j \not = 0$ in \Cref{x0(21)} do not contain an example of an elliptic curve $E/K$ with a $21$-cycle pointwise rational over a quadratic extension of $K$. In the $j=0,1728$ case, Magma is not yet able to describe the isomorphism class, so we instead argue as follows: If there is an elliptic curve $E/K$ such that $E(L)$ has a subgroup of order $21$ for a quadratic extension $L/K$, then by Lemma \ref{lem1} there is a quadratic twist of the original curve $E$ with a $K$-rational point of order 7. But this is not possible by Proposition \ref{cm_case}.

{
\renewcommand{\arraystretch}{1.3}

\begin{table}[ht]
 \caption{($K=\mathbb{Q}(\sqrt{-3})$) Representatives (E,C) of isomorphism classes corresponding to non-cuspidal K-rational points on a model of $X_0(21)$ } 
 
\begin{threeparttable}
 
 \centering 

    \begin{tabular}{ | l | l | l | l |}
    
    \hline
\label{x0(21)}    
    Point & j(E) &  E &  $f_C$ \\ \hline
    
    $(-1/4, 1/8)$ & $3375/2$ & $[20/441, -16/27783 ]$ & $(1,3,3,3)$ \\ \hline
    
   $(2, -1)$  & $-189613868625/128$ & $[-1915/36, -48383/324]$  & $(1,3,6)$ \\ \hline
   
    $(-1, 2)$ &  $-1159088625/2097152$   &   $[-505/192, -23053/6912 ]$ &  $(1,3,6)$  \\ \hline
    
     $(5,13)$  &  $-140625/8$  & $[-1600/147, -134144/9261]$  & $(1,3,3,3)$ \\ \hline

    $(\frac{\pm \sqrt{-3}+1}{2}, \pm \sqrt{-3}-1)$ &  $-12288000$  &  $[\frac{\pm 40\sqrt{-3}+10}{49}, \frac{\mp 2530\sqrt{-3}-6831}{12348}]$ & $(1,3,6)$ \\ \hline
    
    $(\frac{\pm \sqrt{-3}+1}{2}, \frac{\mp 3\sqrt{-3}+1}{2})$  &  $54000$   & $[\frac{\mp 135\sqrt{-3}-585}{98}, \frac{\mp 660\sqrt{-3}-1782}{343} ]$ & $(1,3,6)$  \\ \hline
    
    $(\frac{\mp 3\sqrt{-3}-5}{2}, 8)$  &  0 & See Proposition \ref{cm_case}   & \\ \hline
    
    $(\frac{\mp 3\sqrt{-3}-5}{2}, \frac{\pm 3\sqrt{-3}-11}{2})$  &  0  & See Proposition \ref{cm_case}   & \\ \hline

    \end{tabular}
    \begin{tablenotes}
            \item[$\dagger$] In the last column we list the degrees of the irreducible factors of $f_C$ over $K$.
        \end{tablenotes}
     \end{threeparttable}

\end{table}
}

    Over $K= \Q(i)$, $X_0(21)$ has rank 1. Thus, a search of points in $X_0(21)(K)$ is computationally infeasible to determine whether there is an elliptic curve over $K$ with a $21$-cycle pointwise rational over a quadratic extension of $K$. In order to rule out this case, we will follow a similar approach as in \cite[Proposition 7.4]{Ej18} but using a different tool such as symmetric square of an algebraic curve. 

\begin{prop}\label{21-cycle}
Let $K=\Q(i)$ and let $E/K$ be any elliptic curve. Then $E(L)$ has no subgroup of order $21$ for any quadratic extension $L$ of $K.$ 
\end{prop}


\begin{proof}
Suppose that $E(L)$ has a subgroup of order 21 for a quadratic extension $L/K.$ We may assume that $E(K)$ has a point of order $7$ (by replacing $E$ with a quadratic twist if necessary) by Lemma \ref{lem1}. By the parametrization for the family of elliptic curves with a point of order $7$ in \cite[Table 3, p.217]{Kub76}, $E$ is isomorphic to $E_{t}$ for some $t \neq 0,1$ in $K$ such that
\begin{align}\label{three}
	 E_t : y^2 + (1-c)xy - by = x^3 - bx^2 
\end{align}
with $b = t^3 - t^2$ and $c= t^2 -t.$\footnote{The $j$-invariant of $E$ cannot be $0$ or $1728$ by Proposition \ref{cm_case}. Thus $E$ is a quadratic twist of $E_t$ and if $E(L)$ has a point of order $3$ so has $E_t(L')$ for some quadratic extension $L'/K.$} Now let $\langle P \rangle \subseteq E(L)$ be a cyclic subgroup of order $3$ with $P=(x(P),y(P)).$ Note that $\langle P \rangle$ is Gal$(\overline{K}/K)$-invariant and $x(P) = x(-P).$ Then we have $x(P)^{\sigma} = x(P^{\sigma}) = x(P)$ for any $\sigma \in \Gal(\overline{K}/K)$ and so $x(P)$ is defined over $K$. Hence the pair $(E,P)$ corresponds to the point $(t, x(P))$ on the curve $C$ given by the equation $ \phi(t,x) = 0 $ where $\phi(t,x)$ is the third division polynomial of $E_t$ :

$$ \phi(t,x) = x^4 + (\frac{1}{3}t^4 - 2t^3 + t^2 + \frac{2}{3}t + \frac{1}{3})x^3 + (t^5 - 2t^4 + t^2)x^2 + (t^6 - 2t^5 + t^4)x + (-\frac{1}{3}t^9 + t^8 - t^7 + \frac{1}{3}t^6). $$

It boils down to finding the set $C(K)$ of all $K$-rational points on $C$. By Magma, $C$ is isomorphic (over $\Q) $ to the hyperelliptic curve $\tilde C$ defined by $y^2=f(u)$ where 
$$ f(u)=u^8-6u^6- 4u^5 + 11u^4 + 24u^3 + 22u^2 + 8u +1 $$
outside the singular points: $(0,0)$ and $(0,1)$. By the choice of $t \neq 0,1$, it suffices to determine $\tilde C(K).$ The curve $\tilde C$ has genus 3 and its defining polynomial factors as
$$f(u) = (u^2 + u +1)(u^6 - u^5 -6u^4 + 3u^3 + 14u^2 + 7u +1).$$
Let $\tilde C^{(2)}$ be the symmetric square of $\tilde C.$ The hyperelliptic map $\iota : C \rightarrow \PP^1$ gives a divisor $D_{\infty}$ of degree 2 which consists of the points that map to $\infty \in \PP^1.$ The map $\phi:  \tilde C^{(2)} \rightarrow J $ sending 
$D$ to  $[D-D_{\infty}]$ is injective outside of $0 \in J(\Q)$ and $\phi^{-1}(0)$ is a line. That is 
$ \phi^{-1}(0) = \{ \iota^{-1}(\alpha) : \alpha \in \PP^1(\Q)\}$ and we have
$$ \tilde C^{(2)}(\Q) = \phi^{-1}(0) \cup \phi^{-1}(J(\Q)\backslash \{0\}). $$
By $2$-descent algorithm in Magma, $J(\Q)$ has Mordel-Weil rank 0. As $J$ has a good reduction at 5, the map $J(\Q) \rightarrow J(\mathbb{F}_{5})$ is injective. Since the group generated by a point of order $2$ coming from the factor of degree 2 of $f(u)$ and the difference $\infty_{+} -\infty_{-}$ of the two points at infinity, surjects onto $J(\mathbb{F}_5) \simeq \mathbb{Z}/2\mathbb{Z} \times \mathbb{Z}/52\mathbb{Z}$, it is therefore equal to $J(\Q).$ 

To find $C(K)$, consider a point $P' \in C(K)$ and write $\overline{P'}$ for its image under the nontrivial automorphism of $K.$ Then $D=P+\overline{P'}$ is a divisor of degree 2 which is defined over $\Q$, and so it is a point of $\tilde C^{(2)}(\Q).$ If $\phi(D)=0$ then $D$ is of the form $\iota^{-1}(\alpha)$ for some $\alpha \in \PP^1(\Q).$ This means that $D=(\alpha, \sqrt{f(\alpha)}) + (\alpha, -\sqrt{f(\alpha)})$ and thus $P'=(\alpha, \pm \sqrt{f(\alpha)})$ with $\alpha \in \Q.$ Since $P'$ is $K$-rational, $f(\alpha)=-\beta^2$ for some rational $\beta.$ But $f$ happens to have no real roots, and is thus positive over the real numbers, so cannot be equal to $-\beta^2$ for any rational $\beta.$ Otherwise, $D$ is mapped to a non-zero point in $J(\Q).$ As $J(\Q)$ is finite and isomorphic to $\Z/2\Z \times \Z/52\Z$ by the preceding paragraph, one can list all these 103 elements of $J(\Q)$ in Mumford representations by Magma. We check if these divisors are of degree 2 with the points in the support defined over $K.$ We find that the only such divisors are sums of two rational points, except 
\begin{align*}
	  (\frac{-1 \pm \sqrt{-3}}{2}, 0), & \  (\pm \sqrt{2}, \pm 4\sqrt{2}+5),\ (\pm \sqrt{2}, \mp 4\sqrt{2} - 5),\ (\frac{-2 \pm \sqrt{2}}{2}, \frac{-5 \pm 4\sqrt{2}}{4}), \\
 (1 \pm \sqrt{2}, & 11 \pm 8\sqrt{2}),\ (\frac{-2 \pm \sqrt{2}}{2}, \frac{5 \mp 4\sqrt{2}}{4}), (1\pm \sqrt{2}, -11 \mp 8\sqrt{2}).
\end{align*}
None of these quadratic points is defined over $K.$ This shows that there are no ``exceptional" points over $K$, i.e., no $K$-rational points with $x$-coordinate not in $\Q.$ 
Therefore, 
	$$ \tilde C(K) = \tilde C(\Q)=\{\infty_{+}, \infty_{-},(0,1),(0,-1),(1,1),(1,-1)\}.$$ 
	
This implies that the pair $(E,P)$ associated to $(t,x(P))$ in $C(K)$ corresponds to a point in $\tilde{C}(K)$, so to a rational point on the curve $\tilde{C}$ and therefore to a rational point on $C.$ However, if $E$ is over $\Q$ and $x(P) \in \Q,$ then $P \in E(\Q(\sqrt{d}))$ for some $d \in \Q.$ Then $E(\Q(\sqrt{d}))$ has a subgroup of order 21, but this is impossible by  \cite{LL85}.
\end{proof}

\section*{N=24}

The modular curve $X_0(24)$ is an elliptic curve defined by 
$$y^2=x^3-x^2-4x+4. $$
Over $\mathbb{Q}$, $X_0(24)$ has rank 0, torsion $C_2 \oplus C_4$ and 8 cusps.  The torsion and rank do not grow upon extension to $K = \Q(i)$ or $\Q(\sqrt{-3}).$

\section*{N=27}

The curve $X_0(27)$ is an elliptic curve with model $$y^2 + y = x^3 - 7$$.
Over $\Q$, $X_0(27)$ has rank 0, torsion $C_3$ and 2 cusps.  The torsion and rank do not grow upon extension to $K =  \Q(i).$  The one non-cuspidal point (3,-5) induces a pair $(E,C)$ with $j(E) = -12288000$ and the degrees of the irreducible factors of $f_{C}$ over $\Q(i)$ are (1,3,9).  Because $j(E) \not  = 0,1728$, the isomorphism class of $(E,C)$ just consists of quadratic twists of this pair, and hence will yield the same degrees of irreducible factors. Over $K = \Q(\sqrt{-3})$, $X_0(27)$ has rank 0 and torsion $C_3 \oplus C_3$ with 6 cusps.  As in the case $K=\Q(i)$, the 3 non-cuspidal points do not yield 27-cycles pointwise rational over a quadratic extension of $K$.

\section*{N=30}

Let $K=\Q(i)$ or $\Q(\sqrt{-3})$.  If $E/K$ possesses a cyclic Gal($\overline{K}/K$)-invariant subgroup $C$ of order 30, then it has a unique cyclic subgroup of order 15 that is also $\Gal(\overline{K}/K$)-invariant.  So if $K_C/K$ is degree 1 or 2 
then $E$ possesses a 15-cycle pointwise rational over a quadratic extension of $K$.  
But there are only four such pairs $(E,C')$, and we found that in each case $K_{C'}/K$ was degree 2 so we would have $K_C = K_{C'}$.
But as studied in case $N=15,$ the each torsion over the extension was $C_{15},$ so there are no 30-cycles pointwise rational over a quadratic extension of $K$.

\section*{N=35}

Magma tells us $X_0(35)$ is of genus 3 with affine model $$ y^2 + (-x^4 - x^2 - 1)y = -x^7 - 2x^6 - x^5 - 3x^4 + x^3 - 2x^2 + x $$

Furthermore, Magma found an automorphism of $X_0(35)$ such that the quotient curve $E$ is genus 1 with affine model:
$$y^2 + y = x^3 + x^2 + 9x + 1$$
The quotient map (defined between the projective closures) is given by:

$$\Phi:X_0(35) \rightarrow E $$
$$(x,y,z) \mapsto (p^{\Phi}_1,p^{\Phi}_2,p^{\Phi}_3)$$
where
\begin{align*}
p^{\Phi}_1 &= x^4 - 5x^3z - 8x^2z^2 + 5xz^3 + z^4 \\
p^{\Phi}_2 &= 3x^4 - x^3z + 4x^2z^2 + xz^3 - 7y + 3z^4 \\
p^{\Phi}_3 &= x^4 + 2x^3z - x^2z^2 - 2xz^3 + z^4
\end{align*}
Since $\Phi$ is a rational map, the only potential $K$-rational points of $X_0(35)$ are the non-regular points of $\Phi$ and $\Phi^{-1}(E(K))$. In order to compute $\Phi^{-1}(E(K))$ we must first compute $E(K)$.  Magma/Sage give us the following information:

\begin{table}[ht]
\caption{ $K$-Rational Points on a Genus 1 Quotient of $X_0(35)$}
\begin{center}
    \begin{tabular}{| l | l | l | p{6.5cm} |}
    \hline
    $K$ &  rk($E(K)$) & $E(K)_{tor}$ & Points of $E(K)_{tor}$\\ \hline
    $\mathbb{Q}(i)$ & 0 & $ C_3$ & [0,1,0], [1,3,1],[1,-4,1] \\ \hline
     $\mathbb{Q}(\sqrt{-3})$ & 0 & $C_3 \oplus C_3$ & 
     $[0,1,0], [1,3,1],[1,-4,1], \newline
     [\frac{1}{2}(\pm 5\sqrt{-3}-1), \frac{1}{2}(\mp 5\sqrt{-3}+9),1], \newline
     [\frac{1}{2}(\pm 5\sqrt{-3} - 1), \frac{1}{2}(\pm 5\sqrt{-3}-11),1],\newline
     [-\frac{4}{3}, \frac{1}{18}(\pm 35\sqrt{-3} -9),1 ]$
    \\ \hline
    \end{tabular}
\end{center}
\end{table}

To compute $\Phi^{-1}([x,y,z])$, we form the ideal $<C, p^{\Phi}_1-wx, p^{\Phi}_2-wy,p^{\Phi}_3-wz>$ 
(C denotes the model of $X_0(35)$ above) and compute its Gr{\"o}bner Basis (with respect to the ordering x,y,z,w).  Often, one can find basis elements that allow the system to be solved by hand.  We can assume $z\not = 0$ as the only point on our model of $X_0(35)$ with this property is [0,1,0] and $\Phi$ is not defined at this point.

\begin{table}[ht]
\caption{ Gr{\"o}bner basis data for determination of $\Phi^{-1}(E(K))$ }
\begin{center}
    \begin{tabular}{| l | l | l |  p{10cm} |}
    \hline
    Point P of E(K) & $\Phi^{-1}(P)$ & Gr{\"o}bner basis elements \\ \hline
    [0,1,0] & $\emptyset$  & $w^2$  \\ \hline
    
    [1,3,1] & [0,0,1]  & $xw^2, yw$ \\ \hline
    [1,-4,1] &    [0,1,1]     &  $xzw, yw-w^2,zw^3-w^4$ \\ \hline

    \end{tabular}
\end{center}

\end{table}

For each of the extra six points over $K = \mathbb{Q}(\sqrt{-3})$ the Gr{\"o}bner basis contains a polynomial $g(w)$. Using Magma one can check that in each case the only root of $g(w)$ over $K$ is 0.  Hence the ($K$-rational) inverse image of these points under $\Phi$ is empty.

Finally, using a Gr{\"o}bner basis for the ideal $<C, p^{\Phi}_1, p^{\Phi}_2,p^{\Phi}_3>$ we can determine the non-regular points of $\Phi$.  If $z \not = 0$ then the Gr{\"o}bner basis contains $y^2-6y+4$.  This has no roots over $K$.  Hence the only $K$-rational points on $X_0(35)$ are $[0,0,1]$, $[0,1,0]$ and $[0,1,1]$. However, these points are all cusps.

\section*{N=45}

The curve $X_0(45)$ is  a non-hyperelliptic curve of genus $3$ with a projective model
$$ x^2y^2 +x^3z - y^3z - xyz^2 + 5z^4 $$
The set of all quadratic points on $X_0(45)$ are determined in \cite{OS18} by using a slightly different model. By \cite[Table 8.5]{OS18}, there are no non-cuspidal $K$-rational points on $X_0(45)$ for $K=\Q(i)$ or $\Q(\sqrt{-3}).$ 

We summarize our findings in the following theorem.

\begin{thm}\label{isogeny}
Let $K=\Q(\sqrt{D})$ with $D=-1,-3$, and let $E$ be any elliptic curve defined over $K$. Then $E$ has no $N$-cycles pointwise rational over a quadratic extension of $K$ for $N=20, 21, 24,30, 35, 45.$ In either case, there are exactly four elliptic curves over $K$ (up to isomorphism over $K$) with a $15$-cycle pointwise rational over a quadratic extension of $K$.
\end{thm}

\section{Growth of Torsion}

\subsection{Growth of Cyclic Even-order Torsion}

\begin{lem}\label{max_p}
Let $K=\Q(\sqrt{D})$ with $D=-1,-3$ and $E/K$ an elliptic curve with $E(K)[2]\simeq C_2.$ If $p$ is an odd prime dividing $|E(L)_{\text{tor}}|$ where $L=K(\sqrt{d})$ for $d\in K$ a non-square, then $p \leq 5$ and $ E(L)[p^{\infty}] \simeq C_3, C_5  \ \text{or}\ C_3 \oplus C_3.$
\end{lem}

\begin{proof}
Let $E(L)[p^{\infty}]\simeq C_n$ where $n=p^k$ with $p$ an odd prime and $k \geq 1.$ As the automorphism group Aut$(C_{n})$ is cyclic, the non-trivial element $\sigma$ of Gal$(L/K)$ acts as identity or as multiplication by $-1$ on $E(L)[p^{\infty}].$ So $E$ or its quadratic twist $E^d$ has a $K$-rational point of order $2n.$ This shows $n$ is at most $5,$ i.e. $p \leq 5.$ Moreover, $C_n \oplus C_n \subseteq E(L)[p^{\infty}]$ can only happen if $L$ has a primitive $n$th root of unity $\mu_n$. Let $\varphi$ denote Euler's totient function. It follows that
$ [\Q(\mu_n):\Q]=\varphi(n)=(p-1)p^{k-1} \leq [L : \Q] =4.$
Since $p$ is odd, the inequality holds when $n=3$ or $5.$ 
If $n=5$ then $L=\Q(\mu_5).$ This means that Gal$(L/\Q)$ is cyclic. Hence $\Q(\sqrt{5})$ is the unique intermediate subfield of $L,$ a contradiction since $K \subseteq L.$ Therefore, the largest value of $n$ with $C_n \oplus C_n \subseteq E(L)$ is $n=3.$ 
On the other hand, if $ C_9 \subseteq E(L)[3^{\infty}],$ (this can only happen if $L$ contains $\Q(\sqrt{-3}).$ 
then by Lemma \ref{lem1}, $E(K)$ or $E^d(K)$ has a point of order $18$ which is not possible by Theorem \ref{mainlist}, proving the statement.
\end{proof}

\begin{prop}\label{mainprop}
Let $K=\Q(\sqrt{D})$ with $D=-1,-3,$ $E/K$ any elliptic curve with $E(K)[2]\simeq C_2$ and let $L=K(\sqrt{d})/K$ be a quadratic extension with $d\in K$ non-square.
\begin{enumerate}

\item \label{4x8} $C_4 \oplus C_4 \not\subseteq E(L)$ for $K=\Q(i)$ and $C_4 \oplus C_8 \not\subseteq E(L)$ for $K=\Q(\sqrt{-3}).$

\item \label{6-6x6} If $E(K)_{\text{tor}}\simeq C_6$ then $E(L)_{\text{tor}} \not\simeq C_6 \oplus C_6.$

\item \label{12-2x24} $C_2 \oplus C_{24} \not\subseteq E(L).$

\item \label{30} $C_{N} \not\subseteq E(L)$ for $N=14,18, 20,30,32,48.$ 
\end{enumerate}
\end{prop}

\begin{proof}
$\textit{(1)}$ If $K=\Q(i),$ then the statement follows from \cite[Proposition 10.2]{Ej18}. Let $K=\Q(\sqrt{-3})$ and $C_4 \oplus C_8 \subseteq E(L)$. Then $L=K(i)$ by the Weil Pairing and so the modular curve $X_1(4,8)$ has a non-cuspidal $L$-rational point. However, $X_1(4,8)$ is isomorphic (over $\Q(i)$) to the elliptic curve with the Cremona label 32a2 \cite[Lemma 13]{Naj} and 
$$ X_1(4,8)(L)=X_1(4,8)(\Q(i,\sqrt{-3})) \simeq C_2 \oplus C_4$$
which consists entirely of cusps. Therefore, we get a contradiction. 

$\textit{(2)}$ Let $E(K)_{\text{tor}}\simeq C_6$ and $E(L)_{\text{tor}} \simeq C_6 \oplus C_6.$ By Theorem \ref{burton}\eqref{p-pxp} and the Weil Pairing, $K=\mathbb{Q}(i)$ and  $L=\Q(i,\sqrt{-3}).$ Equivalently, the modular curve $X_1(6,6)$ has a non-cuspidal $L$-rational point. 
Due to \cite[Lemma 14]{Naj}, it has a model over $\mathbb{Q}$ 
such that $ X_1(6,6): y^2 = x^3+1 $
and its cusps satisfy: $x(x - 2)(x + 1)(x^2 - x + 1)(x^2 + 2x + 4) = 0.$ Let $X=X_1(6,6).$ We compute
\begin{align*}
rk(X(L))&=rk(X(\Q(\sqrt{-3}))+rk(X^{(-1)}(\Q(\sqrt{-3})) \\
&= rk(X(\Q))+rk(X^{(-3)}(\Q))+rk(X^{(-1)}(\Q))+rk(X^{(3)}(\Q)) \\
&=0.
\end{align*}
and $X(L)_{\text{tor}} \simeq C_2 \oplus C_6.$ However, all these torsion points are cusps, a contradiction.

$\textit{(3)}$ By the previous assertion, it suffices to show that $E(L)_{\text{tor}}\not\simeq C_2 \oplus C_{24}.$ Let $E(L)_{\text{tor}}\simeq C_2 \oplus C_{24}.$ By Theorem \ref{mainthm1}\eqref{2p-2x4},\eqref{4-2x24} using Lemma \ref{lem1}, we assume $E(K)_{\text{tor}}\simeq C_8$ with a $K$-rational $3$-torsion point $P$ on its quadratic twist $E^d.$ Then $E(K)_{\text{tor}} \oplus \langle P \rangle$ is isomorphic to a Gal$(L/K)$-invariant cyclic subgroup $C$ of $E(L)$ with order 24. But, this is not possible by Theorem \ref{isogeny}.

$\textit{(4)}$ Suppose $C_{N} \subseteq E(L)$ for $N \in \{14,18,20,30,32,48\}$. If $N=14, 18$ then by Theorem \ref{mainthm1}$(ii)$ $E$ or $E^d$ has a $K$-rational point of order 14 and 18 respectively. But this is not possible by Theorem \ref{mainlist}. 

If $N=20,$ then by Theorem \ref{isogeny} using Lemma \ref{n-cycle} we see that $E(L)[20]$ is not cyclic. Furthermore, $E(L)[20]$ contains a subgroup of the form $C_M \oplus C_{20}$ where $M \in \{2,4\}$ by Lemma \ref{max_p}. It then follows from Lemma \ref{lem1} and \ref{generallemma4} that $E$ or $E^d$ has a $K$-rational point of order $20,$ which is a contradiction by \eqref{eq:kamiennygroups}. 

If $N=30,$ then replacing $E$ with its quadratic twist $E^d$ if necessary, we may assume that $E(K)$ has a $6$-torsion point $Q$ and $E^d(K)$ has a $5$-torsion point $P$. Then $C:=\langle Q \rangle \oplus \langle P \rangle$ forms a cyclic Gal$(L/K)$-invariant subgroup of $E(L)_{\text{tor}}$ of order $30,$ which cannot happen due to Theorem \ref{isogeny}. If $N=32,$ then the statement follows from Theorem \ref{mainthm1}\eqref{16point} and \ref{mainlist}. 

If $N=48$ and $E(L)[48]\simeq C_{48}$, then $E(L)[24] \simeq C_{24}$. But $E(L)[24]$ cannot be cyclic by Lemma \ref{n-cycle} and Theorem \ref{isogeny}. If $E(L)[48] \simeq C_{M} \oplus C_{48}$ for $M\neq 1$ dividing 48, then $M=2$ since $C_3 \oplus C_{12}$ and $C_4 \oplus C_{12}$ cannot occur as torsion subgroups over quartic number fields by \cite[Theorem 8]{BN16}. Hence $E(L)[48]\simeq C_2 \oplus C_{48}$ and so it contains the subgroup $C_2 \oplus C_{24}$ which is ruled out by the previous assertion.
\end{proof}
The following table summarizes our results for which we assume $E(K)[2]\simeq C_2.$
\begin{table}[h!]
\centering
\caption{}
\begin{tabular}{ | p{2.9cm}|p{1.3cm}|p{1.3cm}|p{1.3cm}|p{1.3cm}|p{1.3cm}|p{1.3cm}|p{1.3cm}|}
 \hline
$E(L)_{\text{tor}} \backslash E(K)_{\text{tor}}$ & $C_2$  & $C_4$ & $C_6$  & $C_8$ &  $C_{10}$ &  $C_{12}$ &  $C_3 \oplus C_6$   \\
\hline
  $C_2$ & $\checkmark$ & $\rule{.5cm}{0.8pt}$ & $\rule{.5cm}{0.8pt}$ & $\rule{.5cm}{0.8pt}$&$\rule{.5cm}{0.8pt}$&$\rule{.5cm}{0.8pt}$ & $\rule{.5cm}{0.8pt}$  \\
 \hline 
  $C_4$ &$\checkmark$ & $\checkmark$& $\rule{.5cm}{0.8pt}$ & $\rule{.5cm}{0.8pt}$ & $\rule{.5cm}{0.8pt}$  & $\rule{.5cm}{0.8pt}$ & $\rule{.5cm}{0.8pt}$  \\
  \hline
  $C_6$ & $\checkmark$  &$\rule{.5cm}{0.8pt}$ & $\checkmark$  & $\rule{.5cm}{0.8pt}$   & $\rule{.5cm}{0.8pt}$    &  $\rule{.5cm}{0.8pt}$ & $\rule{.5cm}{0.8pt}$  \\
 \hline
 $C_8$ &$\checkmark$    & $\checkmark$  &  $\rule{.5cm}{0.8pt}$   & $\checkmark$    &  $\rule{.5cm}{0.8pt}$     & $\rule{.5cm}{0.8pt}$ &  $\rule{.5cm}{0.8pt}$  \\
 \hline
 $C_{10}$ &$\checkmark$    & $\rule{.5cm}{0.8pt}$   & $\rule{.5cm}{0.8pt}$   &$\rule{.5cm}{0.8pt}$   &$\checkmark$  &$\rule{.5cm}{0.8pt}$ & $\rule{.5cm}{0.8pt}$ \\
 \hline
 $C_{12}$ &$\checkmark$  &$\checkmark$    &$\checkmark$   &$\rule{.5cm}{0.8pt}$   & $\rule{.5cm}{0.8pt}$     &$\checkmark$ &$\rule{.5cm}{0.8pt}$ \\
 \hline
  $C_{14} \not\subseteq E(L)$ & \eqref{30}   & $\rule{.5cm}{0.8pt}$   & $\rule{.5cm}{0.8pt}$  & $\rule{.5cm}{0.8pt}$    &$\rule{.5cm}{0.8pt}$    & $\rule{.5cm}{0.8pt}$ &$\rule{.5cm}{0.8pt}$ \\
 \hline
 $C_{16}$ &$\checkmark$ & \eqref{4-16} & $\rule{.5cm}{0.8pt}$    & $\checkmark$  &$\rule{.5cm}{0.8pt}$   &$\rule{.5cm}{0.8pt}$ & $\rule{.5cm}{0.8pt}$  \\
 \hline
 $C_{18} \not\subseteq E(L)$ & \eqref{30} & $\rule{.5cm}{0.8pt}$  &\eqref{30}   & $\rule{.5cm}{0.8pt}$    &  $\rule{.5cm}{0.8pt}$    & $\rule{.5cm}{0.8pt}$ & $\rule{.5cm}{0.8pt}$ \\
 \hline
 $C_{20} \not\subseteq E(L)$& \eqref{30}  & \eqref{30}   & $\rule{.5cm}{0.8pt}$   & $\rule{.5cm}{0.8pt}$   &\eqref{30}  &$\rule{.5cm}{0.8pt}$ &$\rule{.5cm}{0.8pt}$ \\
 \hline
 $C_{24}$&  \ref{isogeny}  & \ref{isogeny}   & \ref{isogeny}   & \ref{isogeny}   & $\rule{.5cm}{0.8pt}$ &\ref{isogeny} &$\rule{.5cm}{0.8pt}$ \\
 \hline
 $C_{30} \not\subseteq E(L) $ &\eqref{30}  & $\rule{.5cm}{0.8pt}$  &\eqref{30}   &  $\rule{.5cm}{0.8pt}$   &\eqref{30}  & $\rule{.5cm}{0.8pt}$ &$\rule{.5cm}{0.8pt}$ \\
 \hline
 $C_{32} \not\subseteq E(L)$&\eqref{30}  &\eqref{30}  & $\rule{.5cm}{0.8pt}$ &\eqref{30} & $\rule{.5cm}{0.8pt}$  & $\rule{.5cm}{0.8pt}$ & $\rule{.5cm}{0.8pt}$ \\
 \hline
 $C_{48} \not\subseteq E(L)$ & \eqref{30} & \eqref{30} & \eqref{30}   & \eqref{30} & $\rule{.5cm}{0.8pt}$  & \eqref{30} & $\rule{.5cm}{0.8pt}$ \\
 \hline
 $C_2 \oplus C_2$ & $\checkmark$   & $\rule{.5cm}{0.8pt}$ &  $\rule{.5cm}{0.8pt}$  & $\rule{.5cm}{0.8pt}$ &  $\rule{.5cm}{0.8pt}$  &  $\rule{.5cm}{0.8pt}$  & $\rule{.5cm}{0.8pt}$ \\
 \hline
 $C_2 \oplus C_4$ & \eqref{2p-2x4} & $\checkmark$  & \eqref{2p-2x4}  &  $\rule{.5cm}{0.8pt}$  &  $\eqref{2p-2x4}$   & $\rule{.5cm}{0.8pt}$ &  $\rule{.5cm}{0.8pt}$ \\
 \hline
 $C_2 \oplus C_6$ & $\checkmark$ & $\rule{.5cm}{0.8pt}$ &  $\checkmark$ &$\rule{.5cm}{0.8pt}$    & $\rule{.5cm}{0.8pt}$   &  $\rule{.5cm}{0.8pt}$& $\rule{.5cm}{0.8pt}$  \\
 \hline
 $C_2 \oplus C_8$ &\eqref{2p-2x4}   & $\checkmark$    &  $\rule{.5cm}{0.8pt}$  &  $\checkmark$     & $\rule{.5cm}{0.8pt}$  & $\rule{.5cm}{0.8pt}$ & $\rule{.5cm}{0.8pt}$ \\
 \hline
 $C_2 \oplus C_{10}$ & $\checkmark$ & $\rule{.5cm}{0.8pt}$ &  $\rule{.5cm}{0.8pt}$   & $\rule{.5cm}{0.8pt}$ &  $\checkmark$  & $\rule{.5cm}{0.8pt}$& $\rule{.5cm}{0.8pt}$ \\
 \hline
 $C_2 \oplus C_{12}$ & \eqref{2p-2x4} & $\checkmark$  & \eqref{2p-2x4} & $\rule{.5cm}{0.8pt}$ & $\rule{.5cm}{0.8pt}$  & $\checkmark$  & $\rule{.5cm}{0.8pt}$\\
 \hline
 $C_2 \oplus C_{16}$  & \eqref{2p-2x4} & \eqref{8-2x16} & $\rule{.5cm}{0.8pt}$ &  \eqref{8-2x16} & $\rule{.5cm}{0.8pt}$ &  $\rule{.5cm}{0.8pt}$ & $\rule{.5cm}{0.8pt}$ \\
 \hline
 $C_2 \oplus C_{24}$ & \eqref{2p-2x4}  &\eqref{4-2x24} &\eqref{2p-2x4}  & \eqref{12-2x24} & $\rule{.5cm}{0.8pt}$   & \eqref{12-2x24} &  $\rule{.5cm}{0.8pt}$\\
 \hline
 $C_3 \oplus C_6$ & $\eqref{3x3}$, $\checkmark$  & $\rule{.5cm}{0.8pt}$ &  $ $\checkmark$, \eqref{p-pxp}$ & $\rule{.5cm}{0.8pt}$   & $\rule{.5cm}{0.8pt}$ & $\rule{.5cm}{0.8pt}$ &   $\checkmark$ \\
 \hline
 $C_3 \oplus C_{12} \not\subseteq E(L)$ &  \cite{BN16}& \cite{BN16}  & \cite{BN16} & $\rule{.5cm}{0.8pt}$   &  $\rule{.5cm}{0.8pt}$   & \cite{BN16} & \cite{BN16} \\
 \hline
 $C_6 \oplus C_6$ &  $\checkmark$   & $\rule{.5cm}{0.8pt}$   &  \eqref{6-6x6} &  $\rule{.5cm}{0.8pt}$   &  $\rule{.5cm}{0.8pt}$   &   $\rule{.5cm}{0.8pt}$ & $\checkmark$ \\
 \hline
 $C_4 \oplus C_4$ &  \eqref{2p-2x4} & \eqref{4x8}, $\checkmark$ & $\rule{.5cm}{0.8pt}$  &$\rule{.5cm}{0.8pt}$  &$\rule{.5cm}{0.8pt}$  & $\rule{.5cm}{0.8pt}$ &$\rule{.5cm}{0.8pt}$  \\
 \hline
 $C_4 \oplus C_8 \not\subseteq E(L)$ &\eqref{4x8}  & \eqref{4x8}   &  $\rule{.5cm}{0.8pt}$  &\eqref{4x8}   &$\rule{.5cm}{0.8pt}$   & $\rule{.5cm}{0.8pt}$ & $\rule{.5cm}{0.8pt}$\\
 \hline
 $C_4 \oplus C_{12} \not\subseteq E(L) $ & \cite{BN16}  & \cite{BN16}   & \cite{BN16}  & \cite{BN16}  & $\rule{.5cm}{0.8pt}$   & \cite{BN16} & \cite{BN16} \\
\hline
\end{tabular}
\label{table:3.2}
\end{table}

\subsection*{Proof of Theorem \ref{mainthm2}}
The proof follows from the case studies as detailed out in Table \ref{table:3.2} based on Theorem \ref{mainlist}.

\subsection{Cyclic Odd-order Torsion}
\begin{lem}\label{full}
Let $K=\mathbb{Q}(\sqrt{D})$ ($D = -1, -3$), $E/K$ an elliptic curve and $L$ a quadratic extension of $K$. Then the only odd prime power n such that $C_n \oplus C_n  \subseteq E(L) $ is $n=3$.
\end{lem}

\begin{proof}
Let $\phi$ denote Euler's totient function.  If $n=p^t$ where $p$ is a prime and $C_n \oplus C_n  \subseteq E(L) $ then by the Weil pairing, 
$L$ contains an $n$th root of unity $\mu_n$.  Hence $ (p-1)p^{t-1} =  \phi(n) \leq [L:\mathbb{Q}] =4$ so $n=2,3,4,5$ or $8$. Note there is either a 3rd or 4th root of unity in $K$, 
so if $\mu_5$ is in $L$, then there is a 15th or 20th root of unity in $L$. 
But $\phi(15)=8 > 4$ and $\phi(20)=8 >4$, a contradiction.  
On the other hand, there is an elliptic curve $E/\mathbb{Q}$ (namely $[0,-1,1, 217, -282]$) which has full 3-torsion over $K=\mathbb{Q}(\sqrt{-3})$ and hence provides examples in each case with $L=\mathbb{Q}(\sqrt{-1},\sqrt{-3})$.
\end{proof}

\begin{prop}
\label{oddtwist}
Let $K=\mathbb{Q}(\sqrt{D})$ ($D = -1, -3),$  d $ \in$ $K$ a non-square, and $E/K$ an elliptic curve.

\begin{enumerate}
\item If $E(K)_{tor} \simeq C_7, C_9 \mbox{ or }  C_3 \oplus C_3, $ then  $E^d(K)_{tor} \simeq C_1. $

\item If $E(K)_{tor} \simeq C_5 $ then  $E^d(K)_{tor} \simeq C_1 \mbox{ or } C_3. $

\item If $E(K)_{tor} \simeq C_3 $ then  $E^d(K)_{tor} \simeq C_1, C_3 \mbox{ or } C_5. $

\item If $E(K)_{tor} \simeq C_1 $ then  $E^d(K)_{tor} \simeq C_1, C_3, C_5, C_7, C_9 \mbox{ or } C_3 \oplus C_3. $
\end{enumerate}
Hence the torsion structures $C_7, C_9$ and $C_3 \oplus C_3$ do not grow in any quadratic extension of $K$.  The torsion structures $C_3$ and $C_5$ grow in at most 1 extension.
\end{prop}

\begin{proof}
Let $d \in K$ be a non-square.  Note that if $E'$ is a quadratic twist of $E$ then $E$ is a quadratic twist of $E'$ (up to isomorphism over $K$).  Also by Theorem \ref{burton}$(ii)$, all quadratic twists of a curve with odd order torsion will be odd order. By Theorem \ref{mainlist},
the possible nontrivial odd-order torsion structures occurring over $K$ are $C_3,C_5,C_7,C_9$ and $C_3 \oplus C_3$.

Suppose $E(K)[3] \not\simeq C_1.$ Notice that $E^d(K)[3]$ must be cyclic by Lemma \ref{lem1}. If $E^d(K) \simeq C_7$ or $C_9$, then by \cite[Prop 3.1]{Bal(2)21} $E^d$ has a 21-cycle and 27-cycle pointwise rational over $K(\sqrt{d})$, respectively. But no such cycle exists by Theorem \ref{isogeny}. If $E(K)[3] \simeq C_3 \oplus C_3$ and $E^d(K) \simeq C_5$, then by \cite[Prop 3.1]{Bal(2)21} $E^d$ has a 45-cycle pointwise rational over $K(\sqrt{d})$ which is ruled out by Theorem \ref{isogeny}. If $E(K)[3] \simeq C_3$ and $E^d(K) \simeq C_5$, then by \cite[Prop 3.1]{Bal(2)21} $E^d$ has a 15-cycle pointwise rational over a quadratic extension of $K.$ There are four elliptic curves (two pairs of quadratic twists) over $K$ (up to isomorphism over $K$) with such a cycle. For each such curve $E$ (we actually need only check one member of each pair), the factorization of the $3$-division polynomial of $E$ indicates that the nontrivial torsion structures occurring among the quadratic twists of E are $C_3$ and $C_5$ and each occurs exactly once.

If $E(K)_{tor} \simeq C_m$ and $ E^d(K)_{tor} \simeq C_m$ for $m=5,7$, then by Lemma \ref{lem1}, the group $C_m \oplus C_m$ is contained in $E(K(\sqrt{d})$, contradicting Lemma \ref{full}.  
If $E(K)[5]  \simeq C_5$ and $E^d(K) \simeq C_7$ or $C_9$, then by \cite[Prop 3.1]{Bal(2)21} $E^d$ has a 35-cycle or 45-cycle pointwise rational over a quadratic extension, respectively.  But this is not possible by Theorem \ref{isogeny}. 

\end{proof}

\subsection*{Proof of Theorem \ref{burtonmainthm}}
It follows from Proposition \ref{oddtwist} using Theorem \ref{mainlist} and Lemma \ref{lem1}.

\section{Algorithm for Examples of Growth}

Given an elliptic curve $E/K$ with a list of all possible torsion structures over $K$ (as in Theorem \ref{mainlist}), we give an algorithm which computes all quadratic extensions $L/K$ in which $E(K)_{\text{tor}}\subsetneq E(L)_{\text{tor}}$ and describes $E(L)_{\text{tor}}$ in each such case. See the \href{https://github.com/BaseChangeTorsion/The-Effect-of-Quadratic-Base-Change-on-Torsion}{{\bf{code}}} for its implementation. 
We now give an informal description of the algorithm below.
\vspace{.1cm}

Input: $d = -1$ or $-3$ and an elliptic curve $E/K$ where $K = \Q(\sqrt{d}).$
\vspace{.1cm}

Output: List of all quadratic extensions where torsion grows and the torsion structure in each case.

\begin{enumerate}
    \item Let $S := \{2, 3, 5, 7\}$ be the set of all primes $p$ for which there is an elliptic curve $E/K$
with $p\ | \ E(K)_{\text{tor}}$. 

    \item For each prime $p$ in $S$, determine the division polynomial $f_p$ of smallest degree
necessary to detect growth of the $p$-part of $E(K)_{\text{tor}}$ upon quadratic base change:  
 \begin{enumerate}[(a)]
    \item Find the primary decomposition of $E(K)_{\text{tor}}.$ 

    \item For each prime $p$ in $S$, count the number $S_p$ of $p$-summands in the $p$-part of $E(K)_{\text{tor}}.$ 
    
\begin{enumerate}[(i)]
 \item For $p=2$:
    \begin{itemize}
    \item If $S_p=0$ then $E(K)[2^{\infty}]$ does not grow in a quadratic extension by Theorem \ref{burton}. 
    \item If $S_p \neq 0$ and $E(K)[2^{\infty}] \simeq [2^a,2^b]$ with $a<b$ then $f_p=\psi_{2^b}.$
    \end{itemize}
    \item For $p=3$:
    \begin{itemize}
        \item If $S_p=0$ or $1$ then if any growth occurs, by Lemma \ref{lem1}, $E(K)[p]$ must grow by Theorem \ref{burton}, so we let $f_p:=\psi_p.$
        \item If $S_p=2$ then the $p$-part cannot grow in a quadratic extension by Lemma \ref{lem1}, so let $f_p=1.$
    \end{itemize}
        \item For $p>3$: 
        \begin{itemize}
        \item If $S_p=0$ then if any growth occurs, by Lemma \ref{lem1}, $E(K)[p]$ must grow by Theorem \ref{burton}, so we let $f_p:=\psi_p.$
        \item If $S_p=1$ then the $p$-part cannot grow in a quadratic extension by Lemma \ref{full}, so let $f_p=1.$
       \item If $S_p=2$ then the $p$-part cannot grow in a quadratic extension by Lemma \ref{lem1}, so we ignore it.
    \end{itemize}
 \end{enumerate}
\end{enumerate}
      \item For each prime $p$ in $S$:
      \begin{enumerate}[(a)]
      \item Factor $f_p$ over $K.$
      \item For each factor $g$ of $f_p$ over $K$:
      \begin{enumerate}[(i)]
      \item If deg$(g)=1$:
            \begin{itemize}
               \item  We may write $g=x-c.$ Compute $(c,d)$ on $E.$ If $d \not\in K,$ then the torsion grows in $L=K(d)$ and Magma can compute $E(L)_{\text{tor}}.$
        \end{itemize}
                \item If deg$(g)=2$:
                \begin{itemize}
                    \item Construct the splitting field $L$ of $g$ over $K,$ and let $c$ be a root of $g$ in $L.$ Compute $(c,d)$ on $E$. If $d\in L,$ then the torsion grows in $L$ and Magma can compute $E(K)_{\text{tor}}.$
                \end{itemize}   
            \end{enumerate}
         \end{enumerate}
\end{enumerate}

\section*{Appendix}
\begin{table}[h!]
\caption{\label{table:last2}Examples for $K=\Q(i)$ in Case $E(K)[2] \not\simeq C_2 \oplus C_2$}
\centering
\begin{tabular}{| p{2.5cm} | p{2cm} | p{3.3cm} | p{6cm} | }
 \hline
 $E(K)_{\text{tor}}$ & \ \ \ \ \ $d$ & $E(K(\sqrt{d}))_{\text{tor}}$ & Weierstrass Model for E 
 \\
 \hline
 $C_1$ & \ \ \ $-3$ & $C_3$ & $[0,1,1,-769,-8470]$ \\
 \hline
 &\ \ \ \ \ $5$ & $C_5$ & $[1,0,1,549,-2202]$ \\
 \hline
 &\ \ \ $-3$ & $C_7$ & $[1,-1,0,-24,-64]$ \\
 \hline
 &\ \ \ $-3$ & $C_9$ & $ [1,-1,0,-123,-667]$ \\
 \hline
 $C_2$ &\ \ \ \ \  $2$ & $C_4$ &  $[1,0,1,-454,-544] $ \\
 \hline
      & \ \ \ $-3$   & $C_6$&  $[1,0,1,-171,-874]$  \\
 \hline
      & \ \ \ $-3$   & $C_8$  &  $[1,-1,0,0,-5]$ \\
 \hline
  &\ \ \ \ \ $5$   & $C_{10}$ &  $[1,1,0,-700,34000]$\\
  \hline
     & \ \ \ $-3$  & $C_{12}$  &  $[1,0,1,-14,-64]$\\
  \hline
  & \ \ \ $-15$ & $C_{16}$ & $[1,-1,1,47245,-2990253]$ \\
  \hline
  & \ \ \ $-7$  & $C_2 \oplus C_2$ & $[1,0,1,-171,-874]$ \\  
  \hline
& \ \ \ $-3$ & $C_2 \oplus C_6$ & $[0,0,0,0,-27]$  \\
\hline
& \ \ \ $-15$  & $C_2 \oplus C_{10}$  &  $[1,-1,1,-6305,-924303]$ \\
\hline
$C_3$ &\ \ \ $-3$ & $C_3 \oplus C_3$ & $[0,1, 1,-9,-15]$ \\
\hline
&\ \ \ \ \ $5$& $C_3 \oplus C_5$& $[1,0,1,-76,298]$ \\
\hline
$C_4$ & \ \ \ \ \ $5$ & $C_8$ &  $[1,1,1,-80,242]$  \\
\hline
     & \ \ \ $-3$   &  $C_{12}$  & $[1,-1,1,13,-61]$   \\
\hline
&\ \ \ \ $65$ & $C_2 \oplus C_4$& $[1,0,0,-110,435]$  \\
\hline
&\ \ \ $-7$ & $C_2 \oplus C_8$& $[1,1,1,10289,-298411]$ \\
\hline
&\ \ \ $-15$ & $C_2 \oplus C_{12}$ & $[1,1,1,-338,-7969]$      \\
\hline
$C_5$ &\ \ \ \ \ $5$ & $C_3 \oplus C_5$  & $[1,1, 1,-3,1]$ \\
\hline
$C_6$ & \ \ \ \ \ $2$ & $C_{12}$ & $[1,0,1,-289,1862]$ \\
\hline
&\ \ \ \ \ $6$ & $C_2 \oplus C_6$ & $[1,0,1,-289,1862]$   \\
\hline
&\ \ \ $-3$ & $C_3 \oplus C_6$ & $[1,0,1,4,-6]$ \\
\hline
$C_8$ & \ \ \ $-15$  & $C_{16}$ & $[1,0,0,210,900]$    \\
\hline
& \ \ \ \ \ $7$   & $C_2 \oplus C_8$ & $[1,0,0,-39,90]$    \\
\hline
$C_{10}$ & \ \ \ \ $33$& $C_2 \oplus C_{10}$ & $[1,0,0,-45,81]$  \\
\hline
$C_{12}$ & \ \ \ $-15$ & $C_2 \oplus C_{12}$ & $[1,-1,1,-122,1721]$    \\
\hline
\end{tabular}
\end{table}

\newpage
\begin{table}[h!]
\caption{\label{table:last1} Examples for $K=\Q(\sqrt{-3})$ in Case $E(K)[2] \not\simeq C_2 \oplus C_2$}
\centering
\begin{tabular}{| p{2.5cm} | p{2cm} | p{3.3cm} | p{6.75cm} | }
 \hline
 $E(K)_{\text{tor}}$ & \ \ \ \ \ $d$ & $E(K(\sqrt{d}))_{\text{tor}}$ & Weierstrass Model for E 
 \\
 \hline
 $C_1$ &\ \ \ \ \ $5$ &$C_3$ & $[1,1,1,-13,-219]$ \\
 \hline
 &\ \ \ \ \ $5$ & $C_5$ & $[0,-1,1,42,443]$  \\
 \hline
 &\ \ \ \ $13$ & $C_7$ & $[1,0,1,-975873773,11746188793640]$ \\
 \hline
 &\ \ \ \ \ $3$ & $C_9$ & $[0,0,0,-1971,44658]$  \\
 \hline
 &\ \ \   $-15$ & $C_3 \oplus C_3$ &$[0,-1,1,217,-282]$ \\
 \hline
 $C_2$ &\ \ \  $-1$    & $C_4$  &  $[1,1,1,-2160,-39540] $ \\
 \hline
 &\ \ \  $-1$  & $C_6$ & $[0,-1,0,-1,0]$ \\
\hline
&\ \ \ $-1$ & $C_8$ & $[0,-1,0,-384,-2772]$ \\
\hline
&\ \ \ \ \ $5$ & $C_{10}$ & $[1,1,0,-20700,1134000]$ \\
\hline
& \ \ \ \ \ $3$ & $C_{12}$ & $[0,-1,0,24,-144]$ \\
\hline
& \ \ \ $-15$ & $C_{16}$ & $[1,-1,1,47245,-2990253]$ \\
\hline
&\ \ \ \ \ $5$ & $C_2 \oplus C_2$ & $[1,1,1,-2160, -39540]$ \\
\hline
&\ \ \ $-1$ & $C_2 \oplus C_6$ & $[0,-1,0,4,-4]$ \\
\hline
& \ \ \ \ \ $5$ & $C_2 \oplus C_{10}$ & $[1,1,0,-700,34000]$\\
\hline
&\ \ \ \ $21$ & $C_3 \oplus C_6$ & $[1,1,0,-1740,22184]$ \\
\hline
& \ \ \ $-7$ & $C_6 \oplus C_6$ & $[1,1,0,220,2192]$\\
\hline
$C_3$ &\ \ \ \ \ $5$& $C_3 \oplus C_5$ & $[1,0,1,-76,298 ]$ \\
\hline
$C_4$ & \ \ \ \ \ $5$ & $C_8$ & $[1,1,1,-80,242]$ \\
\hline
&\ \ \ $-2$ & $C_{12}$ & $[0,1,0,-4385,94815]$ \\
\hline
&\ \ \ \ $15$ & $C_2 \oplus C_4$ & $[1,1,1,-80,242]$\\
\hline
&\ \ \ \ \ $5$ & $C_2 \oplus C_8$ & $[1,0,1,-1,23]$\\
\hline
& \ \ \ $-15$ & $C_2 \oplus C_{12}$ & $[1,1,1,37,281]$\\
\hline
&\ \ \ $-1$ & $C_4 \oplus C_4$ & $[0,0,0,13,-34]$ \\
\hline
$C_5$ &\ \ \ \ \ $5$ & $C_3 \oplus C_5$ & $[1,1,1,-3,1]$ \\
\hline
$C_{6}$ &\ \ \ $-2$& $C_{12}$  & $[1,0,1,-69,-194]$ \\
\hline
&\ \ \ \  $10$ & $C_2 \oplus C_6$ & $[1,0,1,-69,-194]$\\
\hline
$C_8$ &$6\sqrt{-3}-30$ & $C_{16}$ & $[1,0,0,-39,90]$ \\
\hline
&\ \ \ $-1$ & $C_2 \oplus C_8$ & $[1,1,1,35,-28]$\\
\hline
$C_{10}$ &\ \ \ \ $33$ & $C_2 \oplus C_{10}$ & $[1,0, 0,-45,81]$ \\
\hline
$C_{12}$ &\ \ \ \ \ $5$  & $C_2 \oplus C_{12}$  & $[1, 0,1,1,2]$  \\
\hline
$C_3 \oplus C_6$ &\ \ \ $-7$  & $C_6 \oplus C_6$  & $[1,0,1,4,-6]$ \\
\hline
\end{tabular}
\end{table}

\bibliographystyle{plain}
\bibliography{bibfile}

\end{document}